\documentclass[11pt,a4paper]{article}

\usepackage{epsf,epsfig,amsfonts,amsgen,amsmath,amstext,amsbsy,amsopn,amsthm,cases,listings,color
}
\usepackage{ebezier,eepic}
\usepackage{color}
\usepackage{multirow}
\usepackage{epstopdf}
\usepackage{graphicx}
\usepackage{pgf,tikz}
\usepackage{mathrsfs}
\usepackage[marginal]{footmisc}
\usepackage{enumitem}
\usepackage[titletoc]{appendix}
\usepackage{booktabs}
\usepackage{url}
\usepackage{mathtools}
\usepackage{lineno}
\usepackage[backref=page]{hyperref}
\usepackage{pgfplots}
\usepackage{authblk}
\usepackage{amssymb}
\usepackage{amsmath, amsthm, amssymb, mathrsfs}

\usepackage{subfigure}
\usepackage{mathrsfs}
\usetikzlibrary{arrows}

\allowdisplaybreaks[1]

\definecolor{uuuuuu}{rgb}{0.27,0.27,0.27}
\definecolor{sqsqsq}{rgb}{0.1255,0.1255,0.1255}

\newenvironment{pf}{\noindent{\bf Proof}}

\newtheorem{thm}{Theorem}[section]

\newtheorem{prop}[thm]{Proposition}
\newtheorem{claim}[thm]{Claim}
\newtheorem{lemma}[thm]{Lemma}

\theoremstyle{definition}
\newtheorem{defi}[thm]{Definition}

\begin{document}

\title{\bf\Large Tilings in quasi-random $k$-partite hypergraphs}

\date{}
\author{Shumin Sun\thanks{Shumin Sun is supported by the Warwick Mathematics Institute Centre for Doctoral Training, and gratefully acknowledges funding from the European Research Council (ERC) under the European Union’s Horizon 2020 research and innovation programme (Grant agreement No. 101020255-FDC-ERC-2020-ADG). Email: shumin.sun@warwick.ac.uk}}
\affil{Mathematics Institute,

            University of Warwick,
            
            Coventry, CV4 7AL, UK}
\maketitle
\begin{abstract}
Given $k\ge 2$ and two $k$-graphs ($k$-uniform hypergraphs) $F$ and $H$, an \emph{$F$-factor} in $H$ is a set of vertex-disjoint copies of $F$ that together cover the vertex set of $H$. Lenz and Mubayi 
[\textit{J. Combin. Theory Ser.~B, 2016}]
were first to study the $F$-factor problems in quasi-random $k$-graphs with a minimum degree condition. Recently, Ding, Han, Sun, Wang and Zhou [\textit{J. Combin. Theory Ser.~B, 2023}] gave the density threshold for having all $3$-partite $3$-graphs factors in quasi-random $3$-graphs with vanishing minimum codegree condition $\Omega(n)$. 

In this paper, we consider embedding factors when the host $k$-graph is $k$-partite and quasi-random with partite minimum codegree condition. We prove that if $p>1/2$ and $F$ is a $k$-partite $k$-graph with each part having $m$ vertices, then for $n$ large enough and $m\mid n$, any $p$-dense $k$-partite $k$-graph, with each part having $n$ vertices and partite minimum codegree condition $\Omega(n)$, contains an $F$-factor. We also present a construction showing that $1/2$ is best possible. Furthermore, for $1\leq \ell \leq k-2$, by constructing a sequence of $p$-dense $k$-partite $k$-graphs with partite minimum $\ell$-degree $\Omega(n^{k-\ell})$ having no $K_k(m)$-factor, we show that the partite minimum codegree constraint can not be replaced by other partite  minimum degree conditions. Finally, we prove that $n/2$ is the asymptotic partite minimum codegree threshold for ensuring the existence of  all fixed $k$-partite $k$-graph factors in sufficiently large host $k$-partite $k$-graphs, even in the absence of quasi-randomness.

\end{abstract}
\section{Introduction}

For a positive integer $a$, we denote by $[a]$ the set $\{1,\dots,a\}$. For $k\geq 2$, a \textit{$k$-uniform hypergraph} (in short, \textit{$k$-graph}) $H$ consists of a vertex set $V(H)$ and an edge set $E(H)\subseteq \binom{V(H)}{k}$, that is, every edge is a $k$-element subset of $V(H)$. Let $e(H):=|E(H)|$ be the number of edges in $H$.
For a $k$-graph $H$ and a subset $S\subset \binom{V(H)}{s}$, with $1 \leq s\leq k-1$, let $N_H(S)$ (or $N(S)$) be the set of $(k-s)$-sets $S'\in \binom{V(H)}{k-s}$ such that $S'\cup S\in E(H)$. We call elements of $N_H(S)$ {\emph{neighbors}} of $S$. We define the \emph{degree} of $S$ to be $|N_H(S)|$, denoted by $\deg_H(S)\ (\text{or} \deg(S))$. For a singleton $\{v\}$, we will often write $v$ instead. For a subset $U\subseteq V(H)$, let $H[U]$ be the induced subgraph of $H$ on the vertex set $U$.

A $k$-graph $H$ is \textit{$t$-partite} if there exists a partition of the vertex set $V(H)$ into $t$ parts $V(H)=V_1\cup \cdots \cup V_t$ such that every edge intersects each part in at most one vertex. We say $H$ is \emph{balanced} if $|V_1|=|V_2|=\cdots =|V_t|$. A subset $S\subseteq V(H)$ is said to be \textit{legal} if $|S\cap V_i|\leq 1$ for all $i\in [t]$. In a $k$-partite $k$-graph $H$, we define the \textit{partite minimum $s$-degree} $\delta'_s(H)$ to be the minimum of $\deg_H(S)$ taken over all legal $s$-subsets $S\subseteq V(H)$.
In particular, we call the partite minimum $(k-1)$-degree of $H$ as \textit{partite minimum codegree} of $H$. 

Given two $k$-graphs $F$ and $H$, a \textit{perfect $F$-tiling (or $F$-factor)} in $H$ is a set of vertex-disjoint copies of $F$ that together cover the vertex set of $H$. The study of perfect tilings in graph theory has a long and profound history with a number of results, from the classical results of Corradi--Hajnal~\cite{Corr} and Hajnal--Szemer\'edi~\cite{Hajn} on $K_k$-factors to the famous result of Johansson--Kahn--Vu~\cite{Johansson2008Factors} on perfect tilings in random graphs. One type of perfect tiling problem is under the constraint of the host (hyper)graphs being multipartite. The investigation on this topic has been studied by many researchers~\cite{FI, MarS,LoM,Oleg,AhG,Csa,Kee,KuhD,HanZ,LoK2}.

In this paper, we focus on $F$-factor problem in quasi-random $k$-partite hypergraphs. 
The study of quasi-random graphs was launched in late 1980s by Chung, Graham and Wilson~\cite{Chung-quasi-random}. They proposed several well-defined notions of quasi-random graphs which are equivalent. We note that the $F$-factor issue for quasi-random graphs with positive density and a minimum degree $\Omega(n)$ has been implicitly addressed by Koml\'{o}s--S\'{a}rk\"{o}zy--Szemer\'{e}di~\cite{Koml1997Blow} in the course of developing the famous Blow-up Lemma. Unlike graphs, there are several non-equivalent notions for quasi-random hypergraphs (see~\cite{Reiher_2017}).
One basic notion to define quasi-randomness is uniform edge-distribution which has been studied in~\cite{RRS, LenM}, and this can be applied naturally to multipartite hypergraphs.

\begin{defi}[($p,\mu $){-denseness}] \label{def.}
Given integers $n\geq k\geq2$, let $0<\mu,p <1$, and $H$ be an $n$-vertex $k$-partite $k$-graph with partition $V(H)=V_1\cup \cdots \cup V_k$. We say that $H$ is ($p,\mu $)\emph{-dense} if for all $X_1\subseteq V_1, \dots ,X_k\subseteq V_k$,
\begin{equation}\label{eq:count}
e_H(X_1,\dots,X_k)\geq p|X_1|\cdots|X_k|-\mu n^k,
\end{equation}
where $e_{H}(X_1,\dots,X_k)$ is the number of $(x_1,\dots,x_k)\in X_1\times \cdots \times X_k$ such that $\{x_1,\dots,x_k\}\in E(H)$.
\end{defi}

In particular, we say a $k$-partite $k$-graph $H$ is \emph{$p$-dense} if $H$ is ($p,\mu $)-dense for some small $\mu$.

Lenz and Mubayi~\cite{LenM} were the first to study the $F$-factor problems in quasi-random hypergraphs. They proved that for a linear $k$-graph $F$, every sufficiently large quasi-random $k$-graph with constant density and minimum degree $\Omega(n^{k-1})$ admits an $F$-factor. Ding, Han, Sun, Wang and Zhou~\cite{DHSWZ22} later characterised all $k$-partite guest $k$-graph satisfying this property. Additionally, they~\cite{DHSWZ} established a density threshold of $p=1/8$ for the existence of all $3$-partite $3$-graph factors in quasi-random $3$-graphs under a vanishing minimum codegree condition. In this paper, we explore denseness and partite codegree conditions to ensure $F$-factors in $k$-partite $k$-graphs.

Let $F$ be a $k$-partite $k$-graph with $m$ vertices in each part. We first prove that $p > 1/2$ guarantees an $F$-factor in a $p$-dense $k$-partite $k$-graph $H$ with vanishing partite minimum codegree.

\begin{thm}\label{main}
    Let $k\geq 3$ be an integer. Given $0<\varepsilon, \alpha <1$, and a $k$-partite $k$-graph $F$ with each part having $m$ vertices, there exist an $n_0$ and $\mu >0$ such that the following holds for $n\geq n_0$. If a $(\frac{1}{2}+\varepsilon, \mu)$-dense $k$-partite $k$-graph $H$ with $n$ vertices in each part satisfies that $\delta'_{k-1}(H)\geq \alpha n$ and $n\in m\mathbb{N}$, then $H$ has an $F$-factor.
\end{thm}

Our subsequent construction shows that $1/2$ is the density threshold for containing all balanced $k$-partite $k$-graphs in a $p$-dense $k$-partite $k$-graphs with vanishing partite minimum codegree condition. Let $K_k(m)$ denote the complete $k$-partite $k$-graph with each part containing $m$ vertices. 

\begin{thm}\label{cons1}
    For every $\mu >0$ and integer $m\geq 2$, there exists an $n_0$ such that for all $n\geq n_0$, there exists a $(\frac{1}{2}, \mu)$-dense $k$-partite $k$-graph $H$ with $n$ vertices in each part such that $\delta'_{k-1}(H)\geq (\frac{1}{2}-\mu)n$ and $H$ has no $K_k(m)$-factor.
\end{thm}

One follow-up question is whether a similar density threshold in Theorem~\ref{main} could be achieved under other vanishing partite degree assumptions.
However, the answer appears to be negative. 
Our following result shows that, for every $1\leq \ell \leq k-2$, there exists a $p$-dense $k$-partite $k$-graph $H$ with partite minimum $\ell$-degree $\Omega(n^{k-\ell})$ and $p$ close to $1$, such that $H$ has no $K_k(m)$-factor.

\begin{thm}\label{cons2}
    For every $p\in (0,1)$, $\mu >0$ and integers $k\geq 3$, $1\leq \ell \leq k-2$, $m\geq 2$, there exist an $n_0$ and $\alpha >0$ such that for all $n\geq n_0$, there exists a $(p,\mu)$-dense $k$-partite $k$-graph $H$ with each part having $n$ vertices such that $\delta'_{\ell}(H)\geq \alpha n^{k-\ell}$ and $H$ has no $K_k(m)$-factor.
\end{thm}

Based on the construction in Theorem~\ref{cons1}, Theorem~\ref{main} indicates that in a $p$-dense $k$-partite $k$-graph, if the density $p$ is larger than $1/2$, the partite minimum codegree can be relaxed to vanish while still ensuring all $k$-partite $k$-graph factors. We naturally consider another direction: whether we can relax denseness condition and still guarantee $k$-partite $k$-graph factors, assuming that the partite minimum codegree is larger than $n/2$. Our final result shows that in this case, denseness condition can actually be removed.

\begin{thm}\label{degree}
    Let $k\geq 3$ be an integer. Given $\varepsilon >0$, and a $k$-partite $k$-graph $F$ with each part having $m$ vertices, there exists an $n_0$ such that the following holds for $n\geq n_0$. If a $k$-partite $k$-graph $H$ with each part having $n$ vertices satisfies $\delta'_{k-1}(H)\geq (\frac{1}{2}+\varepsilon)n$ and $n\in m\mathbb{N}$, then $H$ has an $F$-factor.
\end{thm}

\noindent \textbf{Organisation.} The remainder of this paper is organised as follows. In Section 2, we will present probabilistic constructions to prove Theorem~\ref{cons1} and Theorem~\ref{cons2}. Section 3 contains absorbing lemmas, which are the key techniques to prove Theorem~\ref{main} and Theorem~\ref{degree}. We complete the proofs of Theorem~\ref{main} and Theorem~\ref{degree} in Section 4, by providing almost perfect tiling lemmas. Section 5 includes some concluding remarks.

\section{Avoiding $F$-factors}

In this section, we shall prove Theorem~\ref{cons1} and Theorem~\ref{cons2}.

\begin{proof}[Proof of Theorem~\ref{cons1}]
    We prove the theorem by the following construction.
    Let integer $m\geq 2$. For $n\in \mathbb{N}$, define a probability distribution $H(n)$ on $k$-partite $k$-graphs with $n$ vertices in each part as follows. Let $K:= K_{k-1}(n)$ be the complete $(k-1)$-partite $(k-1)$-graph with partition $V(K)=V_1\cup \cdots \cup V_{k-1}$ and $|V_1|=\cdots =|V_{k-1}|=n$. Define a random 2-coloring $\phi:E(K)\longmapsto \{{\color {red}red},~{\color{blue}blue}\}$ where each color is assigned to every edge independently with probability $1/2$. Let $V(H(n))$ := $V(K)\cup V_k$, where $V_k$ is a new part with $n$ vertices. We partition $V_k$ into two parts, $V_{k,1}$ and $V_{k,2}$, where $|V_{k,1}|=\lfloor \frac{n}{2} \rfloor$ if $m \nmid \lfloor \frac{n}{2} \rfloor$, and  $|V_{k,1}|=\lfloor \frac{n}{2} \rfloor-1$ otherwise. The edge set $E(H(n))$ is defined as follows. For a vertex $v\in V_k$ and an edge $e\in E(K)$, make $v\cup e$ into a hyperedge of $H(n)$ when
    
    $\bullet$ $v\in V_{k,1}$ and $e$ has color {\color{red}red}, or

    $\bullet$ $v\in V_{k,2}$ and $e$ has color {\color{blue}blue}.
    ~\\

    Given such a construction, for any $X_1\subseteq V_1, \dots ,X_k\subseteq V_k$, the expectation of $e(X_1,\dots ,X_k)$ is $$\frac{1}{2}|X_1|\cdots|X_{k-1}|\cdot|X_k\cap V_{k,1}|+\frac{1}{2}|X_1|\cdots |X_{k-1}|\cdot|X_k\cap V_{k,2}|=\frac{1}{2}|X_1|\cdots|X_{k-1}|\cdot|X_k|.$$ For any legal $(k-1)$-set $S\in V_1\times \cdots \times V_{k-1}$, the degree of $S$ is at least $\lfloor \frac{n}{2} \rfloor-1$. For any other legal $(k-1)$-set $S$ of $V(H(n))$, the expectation of $\deg (S)$ is $\frac{1}{2}n$. By concentration inequality (e.g. Chernoff's bound) and the union bound, for any $\mu >0$ and sufficiently large $n$, there exists $H\in H(n)$ such that $H$ is $(\frac{1}{2},\mu)$-dense and $\delta'_{k-1}(H)\geq (\frac{1}{2}-\mu)n$.

    We claim that $H$ has no $K_k(m)$-factor. Note that for any $u\in V_{k,1}$ and $v\in V_{k,2}$, there is no legal $(k-1)$-set $S\in V_1\times \cdots \times V_{k-1}$ such that $u\cup S, v\cup S \in E(H)$, otherwise $S$ would receive two colors simultaneously. This implies that every copy of $K_k(m)$ in $H$ must have one part entirely lying in $V_{k,1}$ or $V_{k,2}$. Therefore if $H$ has a $K_k(m)$-factor, then $m\mid |V_{k,1}|$, which contradicts to the condition $m \nmid |V_{k,1}|$.

\end{proof}

We next prove Theorem~\ref{cons2}.

\begin{proof}[Proof of Theorem~\ref{cons2}]
    We use the following construction to prove  Theorem~\ref{cons2}. Let integers $m\geq 2$, $k\geq 3$ and $1\leq \ell \leq k-2$. For $n\in \mathbb{N}$, define a probability distribution $H(n)$ on $k$-partite $k$-graphs with each part having $n$ vertices as follows. Let $G$ be a complete $k$-partite $(\ell+1)$-graph with vertex partition $V(G)=V'_1\cup \cdots \cup V'_{k}$ where $|V'_{1}|=n-1$ and $|V'_{i}|=n$ for $2\leq i\leq k$. Now consider a random 2-coloring $\phi:E(G)\longmapsto \{{\color {red}red},~{\color{blue}blue}\}$ where every edge independently has colour {\color {red}red} with probability $q$ and colour {\color{blue}blue} with probability $1-q$. In particular, we say a subgraph in $G$ is {\color {red}red} (resp. {\color {blue}blue}) if every edge in the subgraph has colour {\color {red}red} (resp. {\color {blue}blue}). We define $V(H(n)):=V_1\cup \cdots \cup V_k$ with $V_1=V'_1\cup v$ and $V_i=V'_i$ for $2\leq i\leq k$, namely we add one new vertex $v$ to the first part of $V(G)$. For a legal $k$-set $S$ of $H(n)$, we make $S$ into a hyperedge of $H(n)$ when
    
    $\bullet$ $S$ does not contain $v$ and $G[S]$ is {\color {red}red}, or

    $\bullet$ $S$ contains $v$ and there exists a {\color {blue}blue} edge in $G[S\setminus v]$.
    ~\\

    For any $X_1\subseteq V_1, \dots ,X_k\subseteq V_k$, the expectation of $e(X_1,\dots ,X_k)$ is at least $$q^{\binom{k}{\ell+1}}(|X_1|-1)|X_2|\cdots |X_k|.$$
    For any legal $\ell$-set $U$, if $U$ does not contain $v$, then the expected value of $\deg (U)$ is at least $q^{\binom{k}{\ell+1}}(n-1)n^{k-\ell-1}$. On the other hand, if $U$ contains $v$, the expected value of $\deg(U)$ is at least $(1-q^{\binom{k-1}{\ell+1}})n^{k-\ell}$. Let $p=q^{\binom{k}{\ell+1}}$ and $\alpha =\frac{1}{2}\min \{q^{\binom{k}{\ell+1}}, 1-q^{\binom{k-1}{\ell+1}}\}$. By probabilistic concentration inequality (e.g. Janson's inequality~\cite{Alon}) and the union bound, for any $\mu >0$ and $n$ large enough, there exists $H\in H(n)$ such that $H$ is $(p,\mu)$-dense and $\delta'_{\ell}(H)\geq \alpha n^{k-\ell}$.

    It remains to show that $H$ has no $K_k(m)$-factor. Suppose not, then let $K'$ be the copy of $K_k(m)$ which covers vertex $v$. As $m\geq 2$, there exists another vertex $u\in V(K')$ such that $u\in V'_1$. Then, there exists a legal $(k-1)$-set $T\in V_2\times \cdots \times V_k$ such that $T\cup v\in E(H)$ and $T\cup u\in E(H)$, which implies that $G[T]$ is {\color {red}red} and $G[T]$ has a {\color {blue}blue} edge respectively. A contradiction.

\end{proof}

\section{The Absorption Technique}

The main tool to prove Theorem~\ref{main} and Theorem~\ref{degree} is the absorbing method. This technique, developed initially by R\"odl, Ruci\'nski and Szemer\'edi~\cite{RRSa}, is a powerful tool in finding spanning structures in graphs and hypergraphs. In this paper, we shall apply a variant of the absorbing method known as the lattice-based absorption method, which was proposed by Han~\cite{Han}. The absorption approach splits the proof into two parts. One is on finding an almost perfect $F$-tiling in the host $k$-partite $k$-graph $H$, and the other is on using absorbing lemma to ``complete" the perfect $F$-tiling. In this section, we present and prove two absorbing lemmas, while postponing finding almost perfect $F$-tilings to Section 4.

Throughout the rest of paper, we use $a\ll b$ to indicate that we select the positive constants  $a, b$ from right to left. More concretely, there is an increasing positive function
$f$ such that, given $b$, whenever we choose some $0<a\le  f(b)$, the subsequent statement holds. Hierarchies of other lengths are defined similarly. 

\subsection{Statements of absorbing lemmas}
Our first absorbing lemma deals with the case when the host $k$-partite $k$-graph is $p$-dense with $p>1/2$ and has vanishing partite minimum codegree, which is the main lemma to prove Theorem~\ref{main}. Given a $k$-partite $k$-graph $H$ with vertex partition $V(H)=V_1\cup \cdots \cup V_k$, we say a vertex subset $S$ is {\emph {balanced}} if $|S\cap V_1|=\cdots=|S\cap V_k|$.

\begin{lemma}[Absorbing Lemma I]\label{Abs1}
Suppose that $1/n\ll \mu, \gamma'\ll \gamma \ll  \varepsilon,\alpha <1$ and $m,k, n \in \mathbb{N}$ with $k\ge 3$.
Let $F$ be a $k$-partite $k$-graph with $m$ vertices in each part. Suppose that a $(\frac{1}{2}+\varepsilon, \mu)$-dense $k$-partite $k$-graph $H$ with each part having $n$ vertices satisfies $\delta'_{k-1}(H)\geq \alpha n$ and $n\in m\mathbb{N}$.
Then there exists a balanced vertex set $W \subseteq V (H)$ with $|W| \leq \gamma n$ such that for any balanced vertex set $U\subseteq V(H)\setminus W$ with $|U|\leq \gamma' n$ and $|U|\in km\mathbb{N}$, both $H[W]$ and $H[W \cup U]$ contain $F$-factors.
\end{lemma}

Our second absorbing lemma is for Theorem~\ref{degree}, which treats the case when the host $k$-partite $k$-graph has large partite minimum codegree without a density condition.

\begin{lemma}[Absorbing Lemma II]\label{Abs2}
Suppose that $1/n\ll \gamma '\ll \gamma \ll  \varepsilon <1$ and $m,k, n \in \mathbb{N}$ with $k\ge 3$.
Let $F$ be a $k$-partite $k$-graph with $m$ vertices in each part. Suppose that a $k$-partite $k$-graph $H$ with each part having $n$ vertices satisfies $\delta'_{k-1}(H)\geq (\frac{1}{2}+\varepsilon) n$ and $n\in m\mathbb{N}$.
Then there exists a balanced vertex set $W \subseteq V (H)$ with $|W| \leq \gamma n$ such that for any balanced vertex set $U\subseteq V(H)\setminus W$ with $|U|\leq \gamma' n$ and $|U|\in km\mathbb{N}$, both $H[W]$ and $H[W \cup U]$ contain $F$-factors.
\end{lemma}

\subsection{Auxiliary lemmas}
To prove Lemma~\ref{Abs1} and Lemma~\ref{Abs2}, we shall state and prove some useful lemmas.

\begin{lemma}\label{cover}
    Suppose that $1/n \ll \eta \ll \alpha <1$ and $m,k,n\in \mathbb{N}$ with $k\geq 3$. Let $F$ be a $k$-partite $k$-graph with $m$ vertices in each part. Suppose that a $k$-partite $k$-graph $H$ with each part having $n$ vertices satisfies $\delta'_{k-1}(H)\geq \alpha n$. Then for any vertex $v\in V(H)$, $v$ is contained in at least $\eta n^{km-1}$ copies of $F$.
\end{lemma}

The proof of Lemma~\ref{cover} bases on the following classical counting result, called {\emph{supersaturation}} initially from Erd\H{o}s~\cite{Erd}.

\begin{prop}\label{super}
     Suppose that $1/n\ll \eta \ll p'<1$ and $f,k,n\in \mathbb{N}$.  Let $F$ be a $k$-partite $k$-graph with $|V(F)|=f$. Suppose that $H$ is a $k$-partite $k$-graph with $|V(H)|=n$ and a vertex partition $V_1\cup \cdots \cup V_k$. If $H$ contains at least $p' n^k$ edges, then $H$ contains at least $\eta n^{f}$ copies of $F$ whose $j$-th part is contained in $V_{j}$ for all $j\in[k]$.
\end{prop}

\begin{proof}[Proof of Lemma~\ref{cover}]
    It is enough to prove Lemma~\ref{cover} for $F=K_k(m)$, as any $F$ in the statement is a subgraph of $K_k(m)$. Suppose that $H$ has vertex partition $V(H)=V_1\cup \cdots \cup V_k$. Without loss of generality, we assume $v\in V_1$. By partite minimum codegree condition, we have $\deg (v)\geq n^{k-2}\cdot \alpha n=\alpha n^{k-1}$. 
    
    Construct an auxiliary $k$-partite $k$-graph $H'$ as follows. Define the vertex set $V(H'):=(V_1\setminus v)\cup V_2\cup \cdots \cup V_k$, and let $E(H'):=\{e\in E(H): v\notin e \ {\text{and}}\ \exists S\in N(v)\ {\text{such that}}\ S\subseteq e\}$, namely we retain edges in $E(H)$ which do not cover $v$ but include some neighbor of $v$ as a subset. Then, $|E(H')|\geq \alpha n^{k-1}\cdot (\alpha n-1)\geq \alpha ^2 n^{k}/2$ as $n$ is sufficiently large. Let $K'$ be a $k$-partite $k$-graph obtained from $K_k(m)$ by removing one vertex $u$ from the first part. Then, by Proposition~\ref{super} with $p'=\frac{\alpha^2}{2k^k}$, there exists some small $\eta $ such that $H'$ contains at least $\eta n^{km-1}$ copies of $K'$. Consider a copy of $K'$ in $H'$, denoted by $K''$. Assume that $V(K'')=V'_1\cup \cdots \cup V'_k$ with $V'_i\subseteq V_i$ for $i\in [k]$. By our construction, every legal $(k-1)$-set $S\in V'_2\times \cdots \times V'_k$ of $K''$ is a neighbor of $v$. Then, we can obtain a copy of $K_k(m)$ in $H$ from $K''$ by embedding $u$ into $v$. Therefore, $v$ is contained in at least $\eta n^{km-1}$ copies of $K_k(m)$.

\end{proof}
  
The following concepts are introduced by Lo and Markstr\"{o}m~\cite{Lo2015}, which would be useful in the subsequent proof. Let $H$ be a $k$-partite $k$-graph with vertex partition $V(H)=V_1\cup \cdots \cup V_k$ and $|V_i|=n$ for $i\in [k]$. Given a $k$-partite $k$-graph $F$ with $m$ vertices in each part, a constant $\beta> 0$, an integer $i \ge 1$ and $j\in [k]$, we say that two vertices $u, v\in V_j$ in $H$ are \textit{$(F, \beta, i)$-reachable} (in $H$) if there are at least $\beta n^{ikm-1}$ $(ikm-1)$-sets $W$ such that both $H[u\cup W]$ and $H[v\cup W]$ contain $F$-factors. In this case, we call $W$ a \emph{reachable set} for $\{u,v\}$.
A vertex subset $U\subseteq V_j$ is said to be \emph{$(F, \beta, i)$-closed} if every two vertices in $U$ are $(F, \beta, i)$-reachable in $H$.
For $v\in V(H)$, denote by $\tilde{N}_{F,\beta, i}(v)$ the set of vertices that are $(F,\beta, i)$-reachable to $v$. 

\begin{lemma}\label{reachable}
Suppose that $1/n \ll \beta \ll \eta$ and $m,k,n\in \mathbb{N}$ with $k\geq 3$. Let $F$ be a $k$-partite $k$-graph with $m$ vertices in each part.
 Suppose that $H$ is a $k$-partite $k$-graph with vertex partition $V_1\cup \cdots \cup V_k$ such that each part has $n$ vertices and every vertex $v\in V(H)$ is contained in at least $\eta n^{km-1}$ copies of $F$. Then for every $j\in [k]$, every set of $\lfloor 1/\eta \rfloor+1$ vertices in $V_j$ contains two vertices that are $(F,\beta ,1)$-reachable in $H$.
\end{lemma}

\begin{proof}
    Set $c:=\lfloor 1/\eta \rfloor$. We choose $\beta $ small enough such that $(c+1)\eta \geq 1+(c+1)^2\beta $. For any vertex $v\in V(H)$, denote by $C_F(v)$ the family of $(km-1)$-sets $W$ such that $H[W\cup v]$ has a copy of $F$. Consider $j\in [k]$ and any $c+1$ vertices $v_1, \dots ,v_{c+1}\in V_j$. As every vertex in $H$ is contained in at least $\eta n^{km-1}$ copies of $F$, we have $\sum_{i=1} ^{c+1}|C_F(v_i)|\geq (c+1)\eta n^{km-1}\geq (1+(c+1)^2\beta )n^{km-1}$. Therefore, by the inclusion-exclusion principle, there exist two vertices $u, v$ such that $|C_F(u)\cap C_F(v)|\geq \beta n^{km-1}$, which implies that there are at least $\beta n^{km-1}$ $(km-1)$-sets $W$ such that both $H[W\cup u]$ and $H[W\cup v]$ have copies of $F$. Namely $u,v$ are $(F,\beta ,1)$-reachable in $H$.
\end{proof}

The following lemma says that if the host $k$-partite $k$-graph $H$ satisfies the conclusion above, that is in each part every constant-sized subset contains two reachable vertices, then we are able to find a large set $S_i$ in each part such that every vertex in $S_i$ has a large reachable neighborhood.

\begin{lemma}\label{S-closed}
Suppose that $\delta ,\beta >0$ and integers $c,m,k,n\in \mathbb{N}$ with $k\geq 3$. Let $F$ be a $k$-partite $k$-graph with $m$ vertices in each part. Let $H$ be a $k$-partite $k$-graph with vertex partition $V_1\cup \cdots \cup V_k$, where each part $V_i$ has $n$ vertices. Suppose that for every $i\in [k]$, every set of $c+1$ vertices in $V_i$ contains two vertices that are $(F,\beta ,1)$-reachable in $H$. Then for every $i\in [k]$, there exists $S_i\subseteq V_i$ with $|S_i|\geq (1-c\delta )n$ such that $|\tilde{N}_{F, \beta , 1}(v)\cap S_i|\geq \delta n$ for every $v\in S_i$.
\end{lemma}

\begin{proof}
For each part $V_i$ with $i\in [k]$, our strategy is iteratively deleting one vertex with few ``reachable neighbors'' and also removing the vertices that are reachable to it from $V_i$.
Set $V_{i}^0:=V_i$. If there is a vertex $v_0\in V_{i}^0$ such that $|\tilde{N}_{F, \beta , 1}(v_0)\cap V_{i}^0|<\delta n$, then let $A_0:=v_0\cup \tilde{N}_{F, \beta , 1}(v_0)$ and let $V_i^1:=V_i^0\setminus A_0$.
Next, we check $V_i^1$ -- if there still exists a vertex $v_1\in V_i^1$ such that $|\tilde{N}_{F, \beta , 1}(v_1)\cap V_i^1|<\delta n$, then let $A_1:=v_1\cup \tilde{N}_{F, \beta , 1}(v_1)$ and let $V_i^2:=V_i^1\setminus A_1$. We repeat the procedure until no such $v_j$ exists.
Suppose that we stop with a set of vertices $v_0,\dots, v_s$.
Note that every two vertices of $v_0,\dots, v_s$ are not $(F, \beta , 1)$-reachable in $V_i$, which implies $s< c$ and $|\bigcup _{0\leq j\leq s}A_j|\leq c\delta n$.
Set $S_i:=V_i\setminus \bigcup _{0\leq j\leq s}A_j $, then $|\tilde{N}_{F, \beta , 1}(v)\cap S_i|\geq \delta n$ holds for every $v\in S_i$, as required.
\end{proof}


The following lemma from~\cite{HT} can provide a partition of each part of $H$ such that every smaller part possesses a good reachable property.

\begin{lemma}{\rm{(}\cite[Theorem 6.3]{HT}\rm{)}}.
\label{partation}
Suppose that $1/n_0\ll \beta_0\ll \beta \ll \delta, 1/c, 1/m <1$ and $c,m,k, n_0 \in \mathbb{N}$ with $k\ge 3$.
 Let $F$ be a $k$-graph on $f$ vertices. Suppose that $H$ is a $k$-graph on $n_0$ vertices, and a subset $S\subseteq V(H)$ satisfies that $|\tilde{N}_{F, \beta , 1}(v)\cap S|\geq \delta n_0$ for any $v\in S$. Further, suppose every set of $c+1$ vertices in $S$ contains two vertices that are $(F, \beta , 1)$-reachable in $H$. Then we can find a partition $\mathcal{P}$ of $S$ into $W_1,\dots,W_r$ with $r \leq \min\{c,\lfloor1/{\delta}\rfloor\}$ such that for any $i \in [r]$, $|W_i|\geq (\delta -\beta )n_0$ and $W_i$ is $(F, \beta_0 , 2^{c-1})$-closed in $H$. 
\end{lemma}

In order to prove the whole $S_i$ in Lemma~\ref{S-closed} is closed, we need the following lemma from~\cite{Han2017Minimum} which gives a sufficient condition to merge different closed parts. Before stating the lemma formally, we introduce the following concepts from Keevash and Mycroft \cite{Ke2015}.  Let $H$ be a $k$-partite $k$-graph with each part having $n$ vertices. Suppose that $\mathcal{P} =\{W_0,W_1,\dots,W_r\}$ is a partition of $V (H)$ with $r\geq k$ which is a refinement of the original $k$-partition of $V(H)$. Let $F$ be a $k$-partite $k$-graph with $f$ vertices. For a subset $S\subseteq V(H)$, the \textit{index vector} of $S$ with respect to $\mathcal{P}$ is the vector 
$$\mathbf{i}_{\mathcal{P}}(S):=(|S\cap W_1|,\dots ,|S\cap W_r|)\in \mathbb{Z}^r.$$
Given a vector $\mathbf{i}\in \mathbb{Z}^r$, we use $\mathbf{i}|_{W_j}$ to denote the coordinate which corresponds to $W_j$. We call a vector $\mathbf{i}\in \mathbb{Z}^r$ an \emph{$s$-vector} if all its coordinates are non-negative and their sum is $s$. Given $\lambda>0$, an $f$-vector $\mathbf{v} \in \mathbb{Z}^r$ is called a \emph{$\lambda$-robust $F$-vector} if at least $\lambda n^f$ copies $F'$ of $F$ in $H$ satisfy $\mathbf{i}_{\mathcal{P}}(V(F'))=\mathbf{v}$. Let $I^{\lambda}_{\mathcal{P},F}(H)$ be the set of all $\lambda$-robust $F$-vectors and $L^{\lambda}_{\mathcal{P},F}(H)$ be the lattice generated by the vectors of $I^{\lambda}_{\mathcal{P},F}(H)$. Let $\mathbf{u}_{W_j}\in \mathbb{Z}^r$ be the {\emph{unit vector}} such that $\mathbf{u}_{W_j}|_{W_j}=1$ and $\mathbf{u}_{W_j}|_{W_i}=0$ for $i\ne j$. 

The following lemma is actually a variant of \cite[\rm Lemma 3.9]{Han2017Minimum} and can be derived directly from the the proof of \cite[\rm Lemma 3.9]{Han2017Minimum}.

\begin{lemma}{\rm{(}\cite[\rm Lemma 3.9]{Han2017Minimum}\rm{)}.}
\label{V-closed}
Let $i_0, k, r,f> 0$ be integers and let $F$ be a $k$-graph with $f:=v(F)$. Given constants $\zeta, \beta_0 , \lambda > 0$, there exists $\beta_0'>0$ and integers $i'_0$ such that the following holds for sufficiently large $n_0$. Let $H$ be a $k$-graph on $n_0$ vertices with a partition $\mathcal{P} =\{W_0,W_1,\dots,W_r\}$ of $V (H)$ such that for each $j\in [r]$, $|W_j|\geq \zeta n_0$ and $W_j$ is $(F,\beta_0 , i_0)$-closed in $H$. If $\mathbf{u}_{W_j} -\mathbf{u}_{W_l} \in  L^{\lambda}_{\mathcal{P},F}(H)$ where $1 \leq  j <l \leq r$, then $W_j\cup W_l$ is $(F, \beta'_0, i'_0)$-closed in $H$. 
\end{lemma}

To use Lemma~\ref{V-closed}, we prove the following lemma, which verifies the conditions in  Lemma~\ref{V-closed}.

\begin{lemma}\label{transferral}
    Suppose that $1/n \ll \mu \ll \lambda \ll \zeta, \varepsilon<1$ and $m,k,n\in \mathbb{N}$ with $k\geq 3$. Let $F$ be a $k$-partite $k$-graph with $m$ vertices in each part. Suppose that $H$ is a $(\frac{1}{2}+\varepsilon, \mu)$-dense $k$-partite $k$-graph with vertex partition $V_1\cup \cdots \cup V_k$ and each part having $n$ vertices. Let $\mathcal{P}_i =\{W_i^0, W_i^1,\dots,W_{i}^{r_i}\}$ be a partition of $V_i$, and let $\mathcal{P}:=\bigcup _{i\in [k]}\mathcal{P}_i$ be a refinement of the original $k$-partition of $V(H)$. Assume that for every $i\in [k]$ and $j\in [r_i]$, $|W_i^{j}|\geq \zeta n$. Then for any $i\in [k]$ and $j_1, j_2\in [r_i]$, we have $\mathbf{u}_{W_i^{j_1}} -\mathbf{u}_{W_i^{j_2}} \in  L^{\lambda}_{\mathcal{P},F}(H)$.
\end{lemma}

\begin{proof}[Proof of Lemma~\ref{transferral}]
    Without loss of generality, we prove the lemma for $i=1$, $j_1=1$ and $j_2=2$. We choose parameters satisfying the following hierarchy:
    $$1/n\ll \mu \ll \lambda \ll \lambda_1,\lambda_2 \ll \zeta ,\varepsilon, 1/m, 1/k.$$
    Set $r':=\sum _{i\in [k]}r_i$. Let $\mathbf{u}\in \mathbb{Z}^{r'}$ be the $(km)$-vector where $\mathbf{u}|_{W_j^1}=m$ for $j\in [k]$ and other coordinates are zero. Let $\mathbf{v}\in \mathbb{Z}^{r'}$ be the $(km)$-vector where $\mathbf{v}|_{W_1^1}=m-1$, $\mathbf{v}|_{W_1^2}=1$, $\mathbf{v}|_{W_j^1}=m$ for $j=2,\cdots,k$ and remaining coordinates are zero. Note that $\mathbf{u}_{W_1^{1}} -\mathbf{u}_{W_1^{2}}=\mathbf{u}-\mathbf{v}$.
    
Let $H_1:=H[W_1^1\cup \cdots \cup W_k^1]$ and $H_2:=H[W^2_1\cup W_2^1\cup \cdots \cup W_k^1]$ be the induced $k$-partite subgraphs of $H$. Since $H$ is $(\frac{1}{2}+\varepsilon,\mu)$-dense, we have 
    $$e(H_1)=e_H(W_1^1,W_2^1,\dots ,W_k^1)\geq (\frac{1}{2}+\varepsilon)|W_1^1|\cdots |W_k^1|-\mu (kn)^k\geq \frac{1}{2}\zeta^kn^k.$$
By Proposition~\ref{super} on $H_1$, there exists a constant $\lambda_1$ such that there are at least $\lambda_1n^{km}$ copies of $F$ with $j$-th part embedded in $W_j^1$ for $j\in [k]$. This implies that $\mathbf{u}\in I^{\lambda}_{\mathcal{P},F}(H)$ as $\lambda\ll \lambda_1$.

It suffices to prove $\mathbf{v}\in I^{\lambda}_{\mathcal{P},F}(H)$, then we have $\mathbf{u}_{W_1^{1}} -\mathbf{u}_{W_1^{2}}=\mathbf{u}-\mathbf{v}\in L^{\lambda}_{\mathcal{P},F}(H)$ as desired. Let $X\subseteq W_1^1$ be the set of vertices $v$ satisfying 
$$\deg_{H_1}(v)\ge (\frac{1}{2}+\frac{\varepsilon}{2})|W_2^1|\cdots |W_k^1|.$$
Similarly, let $Y\subseteq W_1^2$ be the set of vertices $u$ satisfying 
$$\deg_{H_2}(u)\ge (\frac{1}{2}+\frac{\varepsilon}{2})|W_2^1|\cdots |W_k^1|.$$ 
Then we note that for every $x\in X$ and $y\in Y$, the following holds
\begin{equation}\label{eq:intersect}
|N_{H_1}(x)\cap N_{H_2}(y)|\ge \varepsilon |W_2^1|\cdots |W_k^1|.   
\end{equation}
We claim that $|X|\ge \frac{\varepsilon}{2}|W_1^1|$. Indeed, we have 
\[
e(H_1)=e_H(W_1^1,W_2^1,\dots ,W_k^1)\le |X||W_2^1|\cdots |W_k^1|+(\frac{1}{2}+\frac{\varepsilon}{2})(|W_1^1|-|X|)|W_2^1|\cdots |W_k^1|.
\]
On the other hand, by $(\frac{1}{2}+\varepsilon,\mu)$-denseness, $e(H_1)$ is at least 
$$(\frac{1}{2}+\varepsilon)|W_1^1|\cdots |W_k^1|-\mu (kn)^k$$
By double counting and using the fact that every $|W_i^j|\ge \zeta n$ , we derive that
\begin{equation}\label{eq:X-size}
|X|\ge \varepsilon|W_1^1|-\frac{2k^k\mu}{\zeta ^k}\cdot \zeta n\ge \frac{\varepsilon}{2}|W_1^1|.    
\end{equation}
Also, $|Y|\ge \frac{\varepsilon}{2}|W_1^2|$ follows from a similar argument.

Fix a vertex $y\in Y$. Let $H_y$ be the $k$-partite $k$-graph on the vertex set $V(H_y):=X\cup W_2^1\cup\cdots \cup W_k^1$, where $E(H_y)$ consists of all legal $k$-sets $S$ such that $S\in E(H_1)$ and $S$ contains an element of $N_{H_2}(y)$ as a subset. Then, by~(\ref{eq:intersect}) and~(\ref{eq:X-size}), we obtain
\[
e(H_y)=\sum_{x\in X}\deg_{H_y}(x)=\sum_{x\in X}|N_{H_1}(x)\cap N_{H_2}(y)|\ge \varepsilon |X||W_2^1|\cdots |W_k^1|\ge \frac{\varepsilon^2}{2}|W_1^1|\cdots |W_k^1|.
\]
Let $F'$ be the $k$-partite $k$-graph on $km-1$ vertices obtained from $F$ by removing an arbitrary fixed vertex in the first part. By Proposition~\ref{super} on $H_y$, there exists a constant $\lambda_2$ such that there are at least $\lambda_2n^{km-1}$ copies of $F'$ with the first part of $F'$ embedded in $X$ and $j$-th part embedded in $W_j^2$ for $2\le j\le k$. Together with $y\in W_1^2$, each copy of $F'$ in $H_y$ gives a copy $F''$ of $F$ in $H$ satisfying $\mathbf{i}_{\mathcal{P}}(V(F''))=\mathbf{v}$. Since $|Y|\ge \frac{\varepsilon}{2}|W_1^2|$, we would get at least
\[
|Y|\cdot \lambda_2n^{km-1}\ge \frac{\varepsilon}{2}|W_1^2|\cdot \lambda_2n^{km-1}\ge \frac{\varepsilon}{2}\zeta \lambda_2n^{km}\ge \lambda n^{km}
\]
such copies of $F$ in total, which implies $\mathbf{v}\in I^{\lambda}_{\mathcal{P},F}(H)$.

\end{proof}

Given a $k$-partite $k$-graph $F$ with $m$ vertices in each part, a $(km)$-set $S$ and $a\in \mathbb{N}$, we say an $a$-set $A\subseteq V(H)$ is an {\emph{absorbing $a$-set}} for $S$ if $A\cap S=\varnothing$ and both $H[A]$ and $H[A\cup S]$ have $F$-factors. Let $\mathcal{A}_a(S)$ be the family of all absorbing $a$-sets for $S$.
The proofs of two absorbing lemmas depend on the following lemma whose proof idea is similar to the non-multipartite version from~\cite[Lemma 3.6]{Han2017Minimum}.

\begin{lemma} \label{absorb}
Suppose that $1/n \ll \gamma' \ll \eta , \beta_0 ' \ll 1/k, 1/m$, and $i'_0, k,m\in \mathbb{N}$ with $k\geq 3$. Let $F$ be a $k$-partite $k$-graph with $m$ vertices in each part.
Suppose that $H$ is a $k$-partite $k$-graph with vertex partition $V_1\cup \cdots \cup V_k$ and each part has $n$ vertices. Furthermore, $H$ satisfies the following two properties:
\begin{enumerate}[label=$(\roman*)$]
  \item[{\rm (i)}] For every $v\in V(H)$, $v$ is contained in at least $\eta n^{km-1}$ copies of $F$;
  \item[{\rm (ii)}]  For each $V_i$, there exists $V_i^0\subseteq V_i$ with $|V^0_i|\leq \eta^2n$ such that $V_i\setminus V^0_i$ is $(F, \beta'_0, i'_0)$-closed in $H$.
\end{enumerate}
Then there exists a balanced vertex set $W$ with $(\bigcup_{i\in [k]}V_i^0)\subseteq W\subseteq V(H)$ and $|W|\leq  \eta n$ such that for any balanced vertex set $U \subseteq V (H)\setminus W$ with $|U| \leq \gamma' n$ and $|U| \in km\mathbb{N}$, both $H[W]$ and $H[U \cup W]$ contain $F$-factors. 
\end{lemma}

\begin{proof}
    The strategy is as follows. We first show that for any balanced $(km)$-set $S$ in $V(H)\setminus (\bigcup_{i\in [k]}V_i^0)$, there is a robust number of absorbing sets for $S$, which allows us to build an absorbing family $\mathcal{F}_1$ randomly. Then we cover vertices in $(\bigcup_{i\in [k]}V_i^0)\setminus V(\mathcal{F}_1)$ greedily into a family $\mathcal{F}_2$ of copies of $F$. The union $V(\mathcal{F}_1\cup \mathcal{F}_2)$ is the absorbing set we desire.

    Fix a balanced $(km)$-set $S\subseteq V(H)\setminus (\bigcup_{i\in [k]}V_i^0)$. Assume that $S=\bigcup_{i\in [k]}\{v_i^1,\dots ,v_i^m\}$ such that for each $i\in [k]$, $\{v_i^1,\dots ,v_i^m\}$ are $m$ vertices lying in $V_i$. Set $a:=i'_0km(km-1)$ and $\eta_1:=\frac{\eta}{2}\cdot (\frac{\beta'_0}{2})^{km-1}$. We claim that there are at least $\eta_1n^{a}$ absorbing $a$-sets for $S$, namely $|\mathcal{A}_a(S)|\geq \eta_1n^{a}$. First consider a  balanced $(km)$-set $S':=\bigcup_{i\in [k]}\{u_i^1,\dots ,u_i^m\}\subseteq V(H)\setminus (\bigcup_{i\in [k]}V_i^0)$ with $u_1^1=v_1^1$ such that $S\cap S'=v^1_1$ and $S'$ has a copy of $F$. By the property (i) in the statement and each $|V^0_i|\leq \eta^2n$, there are at least 
    $$\eta n^{km-1}-(km-1)n^{km-2}-(k\eta^2n)n^{km-2}\geq \frac{\eta}{2}n^{km-1}$$
    choices of such $S'$, where $(km-1)n^{km-2}$ is the number of $(km)$-sets overlapping one more vertex with $S$ and $(k\eta^2n)n^{km-2}$ is the number of $(km)$-sets sharing some vertex with $\bigcup_{i\in [k]}V_i^0$. For any distinct vertex pair $\{v_i^j, u_i^j\}$ with $i\in [k]$ and $j\in [m]$, they are $(F, \beta'_0, i'_0)$-reachable in $H$ by the property (ii). Therefore there are at least $\beta'_0n^{i'_0km-1}$ reachable $(i'_0km-1)$-sets $S_i^j$ for $\{v_i^j, u_i^j\}$. We aim to find disjoint $S_j^i$ greedily for all  $i\in [k], j\in [m]$ except $i=j=1$ such that $S_j^i$ is also disjoint from $S$ and $S'$. During the selection, there are at most $(2km-1)+(km-1)(i'_0km-1)$ vertices to avoid in each step. Thus there are at least $\beta'_0n^{i'_0km-1}/2$ choices for each $S_i^j$. Let $A$ be the union of all such $S_i^j$ together with $S'\setminus v_1^1$. We claim that $A$ is an absorbing $a$-set for $S$. For $H[A]$, each $S_i^j$ forms an $F$-factor with $u_i^j$ by the reachable property, so $H[A]$ has an $F$-factor. For $H[S\cup A]$, each $S_i^j$ forms an $F$-factor with $v_j^i$ and $S'$ has a copy of $F$, which together give an $F$-factor in $H[S\cup A]$. In total, the number of such absorbing sets $A$ is at least
    $$\frac{\eta}{2}n^{km-1}\cdot (\frac{\beta'_0}{2}n^{i'_0km-1})^{km-1}\geq \eta_1n^a,$$
    namely $|\mathcal{A}_a(S)|\geq \eta_1n^{a}$.

    We next build a family $\mathcal{F}_1$ of balanced $a$-sets randomly. Choose a family $\mathcal{F}$ of balanced $a$-sets by selecting $\binom{n}{a/k}^k$ possible balanced $a$-sets independently with probability $p=\eta_1n^{1-a}/(8a)$. By Chernoff's bound and the union bound, with probability $1-o(1)$ as $n\rightarrow \infty$, the $\mathcal{F}$ satisfies
    \begin{equation}
        |\mathcal{F}|\leq 2p\binom{n}{a/k}^k\leq \frac{\eta_1}{4a}n\quad {\rm and}\quad |\mathcal{A}_a(S)\cap \mathcal{F}|\geq \frac{p}{2}|\mathcal{A}_a(S)|\geq \frac{\eta_1^2}{16a}n \tag{a}
    \end{equation}
    for all balanced $(km)$-set $S$ in $V(H)\setminus (\bigcup_{i\in [k]}V_i^0)$.

    The expected number of pairs of intersecting balanced $a$-sets in $\mathcal{F}$ is at most
    
    $$    \binom{n}{a/k}^k\cdot a \cdot \binom{n}{a/k-1}\binom{n}{a/k}^{k-1}\cdot p^2\leq \frac{\eta_1^2}{64a}n.$$
    Therefore, by Markov's inequality, with probability at least $1/2$, 
    \begin{equation}
        \mathcal{F}\ {\rm has}\ {\rm at}\ {\rm most}\ \frac{\eta_1^2}{32a}n\ {\rm pairs}\ {\rm of}\ {\rm intersecting}\ {\rm balanced}\ a{\text {-sets}}.\tag{b}
    \end{equation}
 Thus there exists a family $\mathcal{F}$ satisfying both (a) and (b). We obtain a subfamily $\mathcal{F}_1$ by removing one balanced $a$-set from intersecting pairs and also removing those $a$-sets which are not absorbing $a$-sets for any balanced $(km)$-set $S$ in $V(H)\setminus (\bigcup_{i\in [k]}V_i^0)$. Then $|V(\mathcal{F}_1)|\leq a|\mathcal{F}_1|\leq a|\mathcal{F}|\leq \eta_1n/4$ and $H[V(\mathcal{F}_1)]$ has an $F$-factor. For any balanced $(km)$-set $S$ in $V(H)\setminus (\bigcup_{i\in [k]}V_i^0)$, we get
 $$|\mathcal{A}_a(S)\cap \mathcal{F}_1|\geq \frac{\eta_1^2}{16a}n-\frac{\eta_1^2}{32a}n\geq \frac{\eta_1^2}{32a}n.$$
 For any balanced vertex set $U \subseteq V (H)\setminus (\bigcup_{i\in [k]}V_i^0\cup V(\mathcal{F}_1))$ with $|U| \leq \gamma' n$ and $|U| \in km\mathbb{N}$, we split $U$ arbitrarily into at most $\gamma'n/(km)$ balanced $(km)$-sets. Since $\eta_1^2n/(32a)\geq \gamma'n/(km)$, for all such balanced $(km)$-sets, we can find disjoint absorbing $a$-sets, which means that $H[U\cup V(\mathcal{F}_1)]$ has an $F$-factor.

 We then greedily pick disjoint copies of $F$ covering vertices in $(\bigcup_{i\in [k]}V_i^0)\setminus V(\mathcal{F}_1)$,  and we denote by $\mathcal{F}_2$ the family of such copies of $F$. Our aim is to avoid vertices belonging to $V(\mathcal{F}_1)$ in this process. For any vertex $v_0\in (\bigcup_{i\in [k]}V_i^0)\setminus V(\mathcal{F}_1)$, there are at least $\eta n^{km-1}$ copies of $F$ containing $v_0$ by the property (i) in the assumption. Furthermore, there are at most $k\eta^2n\cdot km+\eta_1n/4$ vertices to avoid in each step. Thus, we can always find the desired copy of $F$ one by one and finally obtain $\mathcal{F}_2$, since $(k\eta^2n\cdot km+\eta_1n/4)n^{km-2}< \eta n^{km-1}$.

 Let $W=V(\mathcal{F}_1)\cup V(\mathcal{F}_2)$, and $W$ is the desired balanced vertex set with $|W|\leq k\eta^2n\cdot km+\eta_1n/4\leq \eta n$.

\end{proof}

\subsection{Proof of Lemma~\ref{Abs1} and Lemma~\ref{Abs2}}
We are ready to prove Lemma~\ref{Abs1} and Lemma~\ref{Abs2}.

\begin{proof}[Proof of Lemma~\ref{Abs1}]
    We choose parameters in the following hierarchy:
    $$1/n \ll \mu \ll \gamma' \ll \beta'_0 \ll \beta_0, \lambda\ll \zeta, \beta \ll \delta \ll \eta, \gamma \ll \varepsilon ,\alpha ,1/m ,1/k$$
    and $m, n, k\in \mathbb{N}$ with $k\geq 3$. Let $F$ be a $k$-partite $k$-graph with each part having $m$ vertices, and let $H$ be a $(\frac{1}{2}+\varepsilon, \mu)$-dense $k$-partite $k$-graph with each part having $n$ vertices such that $\delta'_{k-1}(H)\geq \alpha n$ and $n\in m\mathbb{N}$ as in the statement. Let the vertex partition of $H$ be $V_1\cup \cdots \cup V_k$. By Lemma~\ref{cover}, every vertex $v\in V(H)$ is contained in at least $\eta n^{km-1}$ copies of $F$. Then by Lemma~\ref{reachable}, for any $i\in [k]$, every set of $\lfloor 1/\eta \rfloor+1$ vertices in $V_i$ contains two vertices that are $(F,\beta ,1)$-reachable in $H$. Set $c:=\lfloor 1/\eta \rfloor$. By Lemma~\ref{S-closed}, for each $i\in [k]$, there exists $S_i\subseteq V_i$ with $|S_i|\geq (1-c\delta )n$ such that $|\tilde{N}_{F, \beta , 1}(v)\cap S_i|\geq \delta n$ for any $v\in S_i$. We then apply Lemma~\ref{partation} to each $S_i$ with $n_0=kn$ and $\delta/k$ in place of $\delta$, and obtain a partition $\{W_i^1,\dots,W_{i}^{r_i}\}$ of $S_i$. Let $\mathcal{P}_i =\{W_i^0, W_i^1,\dots,W_{i}^{r_i}\}$ be the partition of $V_i$ where $W_i^0=V_i\setminus S_i$. Set $\mathcal{P}:=\bigcup _{i\in [k]}\mathcal{P}_i$, which is a refinement of the original $k$-partition of $V(H)$. By Lemma~\ref{transferral}, for any $i\in [k]$ and $j_1, j_2\in [r_i]$, we have $\mathbf{u}_{W_i^{j_1}} -\mathbf{u}_{W_i^{j_2}} \in  L^{\lambda}_{\mathcal{P},F}(H)$. Therefore, following Lemma~\ref{V-closed} with $\mathcal{P}$, we conclude that each $S_i$ is $(F, \beta'_0, i'_0)$-closed. We end with Lemma~\ref{absorb} where $V_i^0=W_i^0=V_i\setminus S_i$, and eventually find the desired set $W$ which possesses good absorbing property as in the statement.
    
\end{proof}

The proof of Lemma~\ref{Abs2} is simpler, where we verify the reachable property in a direct way.
\begin{proof}[Proof of Lemma~\ref{Abs2}]
    The parameters have the following hierarchy:
    $$1/n \ll \gamma' \ll \gamma, \beta, \eta \ll \epsilon, 1/m, 1/k$$
    and $m, n, k\in \mathbb{N}$ with $k\geq 3$. Let $F$ be a $k$-partite $k$-graph with $m$ vertices in  each part, and let $H$ be a $k$-partite $k$-graph with $n$ vertices in each part such that $\delta'_{k-1}(H)\geq (\frac{1}{2}+\varepsilon) n$ and $n\in m\mathbb{N}$. Let the vertex partition of $H$ be $V_1\cup \cdots \cup V_k$. By Lemma~\ref{cover} with $(\frac{1}{2}+\varepsilon)$ in place of $\alpha$, every vertex $v\in V(H)$ is contained in at least $\eta n^{km-1}$ copies of $F$. Thus the property (i) in the Lemma~\ref{absorb} is satisfied. It is sufficient to show that each $V_i$ is $(F,\beta ,1)$-closed. If so, Lemma~\ref{Abs2} follows from Lemma~\ref{absorb} with $i'_0=1$ and $\beta$ in place of $\beta'_0$.

    Without loss of generality, we prove the property for $V_1$. For arbitrary two vertices $u,v\in V_1$, since $\delta'_{k-1}(H)\geq (1/2+\varepsilon) n$, we have $\deg(u), \deg(v)\geq n^{k-2}\cdot (1/2+\varepsilon) n=(1/2+\varepsilon)n^{k-1}$, which implies that $|N(u)\cap N(v)|\geq \varepsilon n^{k-1}$. We construct an auxiliary $k$-partite $k$-graph $H'$ as follows. Let the vertex set $V(H'):=(V_1\setminus\{u, v\})\cup V_2\cup \cdots \cup V_k$, and let the edge set $E(H'):=\{e\in E(H): u,v\notin e \ {\text{and}}\ \exists S\in N(u)\cap N(v)\ {\text{such that}}\ S\subseteq e\}$. Namely we retain edges of $E(H)$ which do not cover $u$ nor $v$ but contain some element of $N(u)\cap N(v)$ as a subset. Since $\delta'_{k-1}(H)\geq (1/2+\varepsilon) n$, we have $|E(H')|\geq \varepsilon n^{k-1}\cdot ((1/2+\varepsilon)n-2)\geq \varepsilon ^2 n^{k}/4$. Let $K'$ be a $k$-partite $k$-graph obtained from $K_k(m)$ with one arbitrary vertex $u'$ removed from the first part. Then, by Proposition~\ref{super}, there exists some small $\beta $ such that $H'$ contains at least $\beta n^{km-1}$ copies of $K'$. By our construction, for each such copy $K''$, both $V(K'')\cup u$ and $V(K'')\cup v$ span copies of $K_k(m)$ with $u'$ embedded to $u$ and $v$ respectively, hence in turn span $\beta n^{km-1}$ copies of $F$, as $F$ is a subgraph of $K_k(m)$.  This means that there are at least $\beta n^{km-1}$ $(km-1)$-sets $W$ such that both $H[u\cup W]$ and $H[v\cup W]$ contain $F$-factors. Therefore, $u,v$ are $(F,\beta ,1)$-reachable, and $V_1$ is $(F,\beta ,1)$-closed.


\end{proof}


\section{Proofs of Theorem~\ref{main} and Theorem~\ref{degree}}

In this section, we provide two almost perfect $F$-tiling lemmas for Theorem~\ref{main} and Theorem~\ref{degree} respectively. These two lemmas, together with the absorbing lemmas, will complete the proofs of Theorem~\ref{main} and Theorem~\ref{degree}. We split them into two separate subsections.

\subsection{Proof of Theorem~\ref{main}}
\begin{lemma}[Almost Perfect Tiling I]\label{Almost-factor}
Suppose that $1/n \ll \mu \ll p,\omega <1$ and $f,k,n\in \mathbb{N}$. Let $F$ be a $k$-partite $k$-graph with $|V(F)|=f$. Suppose that $H$ is a $(p,\mu)$-dense $k$-partite $k$-graph with each part having $n$ vertices. Then there exists an $F$-tiling that covers all but at most $\omega n$ vertices in each part of $H$.
\end{lemma}

\begin{proof}
    For any induced $k$-partite subgraph $H'$ of $H$ with each part having at least $\omega n$ vertices, we have $|E(H')|\geq p\omega^k n^k-\mu (kn)^k$ since $H$ is $(p,\mu)$-dense. By Proposition~\ref{super}, $H'$ contains at least one copy of $F$. Hence, we greedily pick vertex-disjoint copies of $F$ from $H$ until at most $\omega n$ vertices in each part are left.
\end{proof}

\begin{proof}[Proof of Theorem~\ref{main}]
    Suppose that $1/n\ll \mu \ll \gamma' \ll \gamma \ll \varepsilon,\alpha<1$ and $m,k,n\in \mathbb{N}$ with $k\geq 3$. Let $F$ be a $k$-partite $k$-graph with $m$ vertices in each part, and let $H$ be a $(\frac{1}{2}+\varepsilon, \mu)$-dense $k$-partite $k$-graph with $n$ vertices in each part such that $\delta'_{k-1}(H)\geq \alpha n$ and $n\in m\mathbb{N}$. By Lemma~\ref{Abs1}, there exists a balanced vertex set $W \subseteq V (H)$ with $|W| \leq \gamma n$ such that for any balanced vertex set $U\subseteq V(H)\setminus W$ with $|U|\leq \gamma' n$ and $|U|\in km\mathbb{N}$, both $H[W]$ and $H[W \cup U]$ contain $F$-factors. Let $H':=H[V(H)\setminus W]$ be the induced subgraph of $H$. Note that $H'$ is $(\frac{1}{2}+\varepsilon,\mu_1)$-dense $k$-parite $k$-graph for some $\mu_1$, since $$\mu (kn)^k= \mu\cdot \big(\frac{k}{k-\gamma}\big)^k\cdot (k-\gamma)^kn^k\leq \mu \cdot (\frac{k}{k-\gamma})^k|V(H')|^k=\mu_1|V(H')|^k.$$
    Applying Lemma~\ref{Almost-factor} on $H'$ with $\omega =\gamma '/k$, we obtain an $F$-tiling that covers all but a balanced set $U$ of at most $\gamma' n$ vertices. By the absorbing property of $W$, $H[W\cup U]$ contains $F$-factor, which gives an $F$-factor of $H$.

\end{proof}

\subsection{Proof of Theorem~\ref{degree}}


\begin{lemma}[Almost Perfect Tiling II]\label{almost2}
    Suppose that $1/n \ll \omega ,\varepsilon<1$ and $m,k,n\in \mathbb{N}$. Let $F$ be a $k$-partite $k$-graph with $m$ vertices in each part. Suppose that a $k$-partite $k$-graph $H$ with each part having $n$ vertices satisfies $\delta'_{k-1}(H)\geq (\frac{1}{2}+\varepsilon)n$. Then there exists an $F$-tiling covering all but at most $\omega n$ vertices in each part of $H$.
\end{lemma}

The proof of Lemma~\ref{almost2} depends on the so-called {\emph{weak hypergraph regularity lemma}}, a straightforward generalization of Szemer\'edi regularity lemma for graphs.

\subsubsection{Weak regularity lemma for hypergraphs}

Let $H$ be a $k$-graph. Given $k$ pairwise disjoint subsets $A_1\dots ,A_k\subseteq V(H)$, we define the {\emph{density}} of $H$ with respect to $(A_1,\dots,A_k)$ as
$$d_H(A_1,\dots ,A_k)=\frac{e_H(A_1,\dots ,A_k)}{|A_1|\cdots |A_k|}.$$
Given $\varepsilon>0$ and $d\geq 0$, a $k$-tuple $(V_1,\dots ,V_k)$ of mutually disjoint subsets $V_1,\dots ,V_k\subseteq V(H)$ is called $(\varepsilon,d)$-{\emph{regular}}, if for all $A_i\subseteq V_i$ with $|A_i|\geq \varepsilon|V_i|$, $i\in [k]$, we have 
$$|d_H(A_1,\dots, A_k)-d|\leq \varepsilon.$$
We say a $k$-tuple $(V_1,\dots ,V_k)$ is $\varepsilon$-{\emph{regular}} if it is $(\varepsilon,d)$-regular for some $d\geq 0$.

A straightforward extension of the graph regularity lemma is given as follows, which was proved by Chung~\cite{weakc}.

\begin{lemma}[Weak regularity lemma for hypergraphs]\label{weak-reg}
For all integers $k\geq 2$ and $t_0\geq 1$, and every $\varepsilon>0$, there exists $T_0$ and $n_0$ such that the following holds. For every $k$-uniform hypergraph $H$ on $n\geq n_0$ vertices, there exists a constant $t\in \mathbb{N}$ with $t_0\leq t\leq T_0$ and a partition $V(H)=V_0\cup V_1\cup \cdots \cup V_t$ such that 

$({\rm i})$ $|V_1|=|V_2|=\cdots =|V_t|$ and $|V_0|\leq \varepsilon n$,

$({\rm ii})$ for all but at most $\varepsilon t^k$ sets $\{i_1,\dots ,i_k\}\in \binom{[t]}{k}$, the $k$-tuple $(V_{i_1},\dots ,V_{i_k})$ is $\varepsilon$-regular.

\end{lemma}

A partition as given in the Lemma~\ref{weak-reg} is called an $\varepsilon$-{\emph{regular partition}} of $H$. We call $V_1,\dots ,V_t$ in the above lemma {\emph{clusters}}. For an $\varepsilon$-regular partition $\mathcal{P}=\{V_0,V_1,\dots ,V_t\}$ of $H$, the {\emph{cluster hypergraph}} $\mathcal{R}=\mathcal{R}(\varepsilon ,d)$ is defined with vertex set $[t]$ and a $k$-tuple $\{i_1,\dots, i_k\}\in \binom{[t]}{k}$ forming an edge if and only if $(V_{i_1},\dots,V_{i_k})$ is $\varepsilon$-regular and $d_H(V_{i_1},\dots,V_{i_k})\geq d$.

\subsubsection{Proof of Lemma~\ref{almost2}}
\begin{proof}[Proof of Lemma~\ref{almost2}]
We pick constants with the following hierarchy:
$$1/n \ll \xi', \xi, \varepsilon_0, \varepsilon^*\ll d\ll \omega, \varepsilon, 1/m, 1/k.$$
Let $H$ be a $k$-partite $k$-graph with each part having $n$ vertices such that $\delta'_{k-1}(H)\geq (\frac{1}{2}+\varepsilon)n$. Note that an $\varepsilon$-regular partition can be obtained by iterated refinements starting with an arbitrary initial partition of $H$. Applying Lemma~\ref{weak-reg} on $H$ with $\varepsilon_0$ and the original partition $V(H)=V_1\cup \cdots \cup V_k$, we obtain the given $\varepsilon_0$-regular partition $\mathcal{P}=V'_0\cup (\bigcup _{i\in [k]}\mathcal{P}_i)$, where $\mathcal{P}_i$ partitions each $V_i$ into $t_i$ clusters for every $i\in [k]$. Assume that each cluster except $V'_0$ has $m_0$ vertices. By removing at most $k^2\varepsilon_0 n/m_0$ clusters into $V'_0$ if necessary, we can obtain a new partition $\mathcal{P}'=V_0\cup (\bigcup _{i\in [k]}\mathcal{P}'_i)$ where $\mathcal{P}'_i\subseteq \mathcal{P}_i$ splits each $V_i$ into exactly $t$ clusters with $t=\min\{t_1,\dots,t_k\}$. By choosing $\varepsilon_0$ small enough, $\mathcal{P}'$ is a $(k\varepsilon_0)$-regular partition of $H$. Set $\varepsilon^*:=k\varepsilon_0$, and let $\mathcal{R}=\mathcal{R}(\varepsilon^*,d)$ be the corresponding cluster hypergraph. Note that $\mathcal{R}$ is a $k$-partite $k$-graph with each part having $t$ vertices.

The next proposition shows that the cluster hypergraph inherits the partite minimum codegree condition of $H$.

\begin{prop}\label{inherit}
 Suppose that $1/n \ll \xi ,\varepsilon^* \ll d\ll \varepsilon<1$. Let $H$ be a $k$-partite $k$-graph with each part having $n$ vertices such that $\delta'_{k-1}(H)\geq (\frac{1}{2}+\varepsilon)n$. Suppose that $\mathcal{P'}$ is an $\varepsilon^*$-regular partition and the corresponding cluster hypergraph $\mathcal{R}=\mathcal{R}(\varepsilon^*,d)$ is a $k$-partite $k$-graph with each part having $t$ vertices. Then the number of legal $(k-1)$-sets $S$ violating 
$$\deg_{\mathcal{R}}(S)\geq (\frac{1}{2}+\frac{\varepsilon}{4})t$$
is at most $\xi t^{k-1}$.

\end{prop}

\begin{proof}
    Assume that $\mathcal{P}'=\{V_0, W_1, W_2,\dots ,W_{kt}\}$. The cluster hypergraph $\mathcal{R}$ can be viewed as the intersection of two $k$-partite $k$-graphs $\mathcal{D}=\mathcal{D}(d)$ and $\mathcal{G}=\mathcal{G}(\varepsilon^*)$. Both of them have the same vertex set $[kt]$ and

    $\bullet$ $\mathcal{D}(d)$ consists of all legal sets $\{i_1,\dots ,i_k\}$ with $d(W_{i_1},\dots ,W_{i_k})\geq d.$

    $\bullet$ $\mathcal{G}(\varepsilon^*)$ consists of all legal sets $\{i_1,\dots ,i_k\}$ with $(W_{i_1},\dots ,W_{i_k})$ being $\varepsilon^*$-regular.

    Given any legal $(k-1)$-set $S=\{i_1,\dots ,i_{k-1}\}\subseteq V(\mathcal{R})$, we first show that $\deg_{\mathcal{D}}(S)\geq (\frac{1}{2}+\frac{\varepsilon}{2})t$. Consider the $(k-1)$-tuple $(W_{i_1},\dots ,W_{i_{k-1}})$ corresponding to $S$ with $m_0:=|W_{i_j}|\leq n/t$. By the partite minimum codegree condition of $H$, we have the number of edges 
    $$e_H(W_{i_1},\cdots ,W_{i_{k-1}}, V(H)\setminus V_0)\geq m_0^{k-1}\cdot (\frac{1}{2}+\varepsilon-k\varepsilon^*)n.$$
    Suppose on the contrary that $\deg_{\mathcal{D}}(S)< (\frac{1}{2}+\frac{\varepsilon}{2})t$. Then,
    $$e_H(W_{i_1},\cdots ,W_{i_{k-1}}, V(H)\setminus V_0)<(\frac{1}{2}+\frac{\varepsilon}{2})tm_0^k+tdm_0^k\leq m_0^{k-1}\cdot (\frac{1}{2}+\frac{\varepsilon}{2}+d)n,$$
    which gives a contradiction, since we can select $d+k\varepsilon^*\leq \varepsilon/2$.

    On the other hand, since there are at most $k^k\varepsilon^*t^k$ irregular $k$-tuples, by double counting, the number of legal $(k-1)$-sets $S$ such that $\deg_{\mathcal{G}}(S)<(1-\sqrt{\varepsilon^*})t$ is at most

    $$\frac{k\cdot k^k\varepsilon^*t^k}{\sqrt{\varepsilon^*}t}:=\xi t^{k-1}.$$

    Since $\mathcal{R}=\mathcal{D}\cap \mathcal{G}$, for those legal $(k-1)$-sets $S$ satisfying both $\deg_{\mathcal{D}}(S)\geq (\frac{1}{2}+\frac{\varepsilon}{2})t$ and $\deg_{\mathcal{G}}(S)\geq (1-\sqrt{\varepsilon^*})t$, we have $\deg_{\mathcal{R}}(S)\geq (\frac{1}{2}+\frac{\varepsilon}{2}-\sqrt{\varepsilon^*})t$. Hence the proposition follows.

\end{proof}

The following theorem from~\cite{KuhD} guarantees the existence of  an almost perfect matching in the cluster hypergraph. 

\begin{thm}[\rm{~\cite[Theorem 11]{KuhD}}]\label{almostPM}
    Let $t_0$ be an integer and $R$ be a $k$-partite $k$-graph with each part having $t$ vertices. Put

\[
\delta':= \begin{cases}
\lceil t/k \rceil&\text{if\ } 
t\equiv 0\bmod k\text{\ or\ }t\equiv k-1\bmod k\\
\lfloor t/k \rfloor&\text{otherwise}.
\end{cases}
\]
Suppose that there are fewer than $t_0^{k-1}$ legal $(k-1)$-sets $S$ satisfying $\deg_R(S)< \delta'$. Then $R$ has a matching which covers all but at most $(k-1)t_0-1$ vertices in each part of $R$.

\end{thm}

We apply Theorem~\ref{almostPM} with the cluster hypergraph $\mathcal{R}$ and $t_0:=\lfloor \xi^{\frac{1}{k-1}}t\rfloor$ and obtain an almost perfect matching of $\mathcal{R}$ covering all but at most $\xi't$ vertices in each part, where $\xi':=(k-1)\xi^{\frac{1}{k-1}}$. Each edge of this matching corresponds to an $(\varepsilon^*,d)$-regular $k$-tuple in $H$.

The next claim says that we are able to find almost $F$-tiling in an $(\varepsilon^*,d)$-regular $k$-tuple.

\begin{claim}\label{greedy}
   Suppose that $\varepsilon^*\ll d$ and $m_0$ is sufficiently large. Suppose that $(W_1,\dots ,W_k)$ is $(\varepsilon^*,d)$-regular and each $|W_i|=m_0$ for $i\in [k]$. Then there exists an $F$-tiling on $W_1\cup \cdots \cup W_k$ covering all but at most $\varepsilon^* m_0$ vertices in each $W_i$.
\end{claim}

\begin{proof}
    Since $(W_1,\dots ,W_k)$ is $(\varepsilon^*,d)$-regular, then for every $A_i\subseteq W_i$ with $|A_i|\geq \varepsilon^*|W_i|=\varepsilon m_0$ and $i\in [k]$, we have $e_H(A_1,\dots ,A_k)\geq (d-\varepsilon^*)|A_1|\cdots |A_k|$. By Proposition~\ref{super}, $H[A_1\cup \cdots \cup A_k]$ contains at least one copy of $F$. Thus, we greedily pick vertex-disjoint copy of $F$ from $W_1\cup \cdots \cup W_k$ until at most $\varepsilon^* m_0$ vertices left in each part.
\end{proof}

Claim~\ref{greedy} implies that for every edge in the almost perfect matching of $\mathcal{R}$, there is an almost $F$-tiling in the corresponding $k$-tuple. Overall, we finally get an $F$-tiling of $H$ covering all but at most
$$k\varepsilon_0n+
\xi'tm_0+\varepsilon^*m_0t\leq (k\varepsilon_0+\xi'+\varepsilon^*)n\leq \omega n$$
vertices in each part of $H$.

\end{proof}

\subsubsection{Proof of Theorem~\ref{degree}}

We now prove Theorem~\ref{degree}.

\begin{proof}[Proof of Theorem~\ref{degree}]
    Suppose that $1/n \ll \omega ,\gamma' \ll \gamma \ll \varepsilon<1$ and $m,k,n\in \mathbb{N}$ with $k\geq 3$. Let $F$ be a $k$-partite $k$-graph with $m$ vertices in each part, and let $H$ be a $k$-partite $k$-graph with each part having $n$ vertices such that $\delta'_{k-1}(H)\geq (\frac{1}{2}+\varepsilon) n$ and $n\in m\mathbb{N}$. By Lemma~\ref{Abs2}, there exists a balanced vertex set $W \subseteq V (H)$ with $|W| \leq \gamma n$ such that for any balanced vertex set $U\subseteq V(H)\setminus W$ with $|U|\leq \gamma' n$ and $|U|\in km\mathbb{N}$, both $H[W]$ and $H[W \cup U]$ contain $F$-factors. Let $H':=H[V(H)\setminus W]$ be the induced subgraph of $H$. Note that $H'$ is a balanced $k$-parite $k$-graph with each part having at least $(1-\gamma/k)n$ vertices and
    $$\delta'_{k-1}(H')\geq (\frac{1}{2}+\varepsilon -\frac{\gamma}{k})n.$$
    Applying Lemma~\ref{almost2} on $H'$ with $\omega =\gamma '/k$ and $\varepsilon -\frac{\gamma}{k}$ in place of $\varepsilon$, we obtain an $F$-tiling that covers all but a balanced set $U$ of at most $\gamma' n$ vertices. By the absorbing property of $W$, $H[W\cup U]$ contains $F$-factor, which together give an $F$-factor of $H$.
\end{proof}

\section{Concluding Remarks}
In this paper, we focus on the density condition and partite minimum codegree condition for embedding balanced factors. Note that our arguments with minor adjustments actually apply to the non-balanced case, and we are able to obtain that the density threshold for embedding all non-balanced $k$-partite $k$-graphs is also $1/2$ under the condition that the host $k$-partite $k$-graph possesses the corresponding divisibility assumption and vanishing partite minimum codegree.

We also note that Mycroft gave the non-multipartite version of Theorem~\ref{degree} in~\cite[Theorem 1.1]{My}~(including non-balanced case) and in particular showed that the asymptotic minimum codegree threshold of balanced complete $k$-partite $k$-graphs in non-multipartite host $k$-graphs is $n/2$. Our results in Theorem~\ref{cons1} and Theorem~\ref{degree} gave the same value of asymptotic partite minimum codegree threshold of balanced complete $k$-partite $k$-graphs when the host $k$-graphs are $k$-partite, which in fact implies the threshold in non-multipartite host $k$-graphs by considering a random partition into $k$ parts.

~\

\noindent\textbf{Acknowledgements.} The author would like to thank Oleg Pikhurko for proposing this problem, also for helpful discussions and writing suggestions.

\bibliographystyle{abbrv}
\bibliography{ref}
\end{document}